\documentclass[a4paper,reqno]{amsart}

\usepackage{amssymb}
\usepackage{latexsym}
\usepackage{amsmath}
\usepackage{euscript}
\usepackage{bbm}

\def\sa{{\mathfrak a}}

   \def\cH{{\mathcal H}}   
      \def\cL{{\mathcal L}}
      \def\cO{{\mathcal O}}

\def\cal H{{\mathcal H}}

\def\R{\mathbb{R}}
\def\C{\mathbb{C}}

\def\N{\mathbb{N}}

\def\dom{{\text{\rm dom\,}}}
\def\phi{\varphi}
\def\d{\textup{d}}

\DeclareMathOperator{\Real}{Re}

\DeclareMathOperator{\spann}{span}

\renewcommand{\theta}{\vartheta}

\newtheorem{theorem}{Theorem}[section]
\newtheorem*{thm*}{Theorem}
\newtheorem{proposition}[theorem]{Proposition}

\newtheorem{lemma}[theorem]{Lemma}
\newtheorem{assumption}[theorem]{Assumption}

\theoremstyle{definition}

\numberwithin{equation}{section}

\title[]{Strict inequality of Robin eigenvalues for elliptic differential operators on Lipschitz domains}

\author{Jonathan Rohleder}

\address{Institut f\"ur Numerische Mathematik \\
Technische Universit\"at Graz \\
Steyrergasse 30\\
A-8010 Graz\\
Austria}
\email{rohleder@tugraz.at}

\begin{document}

\begin{abstract}
On a bounded Lipschitz domain we consider two selfadjoint operator realizations of the same second order elliptic differential expression subject to Robin boundary conditions, where the coefficients in the boundary conditions are functions. We prove that inequality between these functions on the boundary implies strict inequality between the eigenvalues of the two operators, provided that the inequality of the functions in the boundary conditions is strict on an arbitrarily small nonempty, open set.
\end{abstract}

\maketitle

\section{Introduction}

We consider an elliptic differential expression of second order of the form
\begin{align}\label{eq:diffExprIntro}
 \cL = - \sum_{j,k = 1}^n \partial_j a_{jk} \partial_k + \sum_{j = 1}^n \big( a_j \partial_j - \partial_j \overline{a_j} \big)  + a
\end{align}
with bounded Lipschitz coefficients on a bounded, connected Lipschitz domain \linebreak $\Omega \subset \R^n$, $n \geq 2$; see Assumption~\ref{ass:A} below. Given two real-valued functions $\theta_1, \theta_2 \in L^p (\partial \Omega)$ (for appropriate $p$, see Assumption~\ref{ass:B} below) with 
\begin{align}\label{eq:ineq}
\theta_1 \leq \theta_2 \quad \text{on}~\partial \Omega
\end{align}
we focus on the purely discrete spectra of the selfadjoint operators associated with $\cL$ in $L^2 (\Omega)$ subject to the Robin boundary conditions
\begin{align}\label{eq:Robin}
 \frac{\partial u}{\partial \nu_\cL} \big|_{\partial \Omega} + \theta_j u |_{\partial \Omega} = 0, \quad j = 1, 2;
\end{align}
here $u |_{\partial \Omega}$ denotes the trace and $\frac{\partial u}{\partial \nu_\cL} |_{\partial \Omega}$ is the conormal derivative of $u$ at the boundary $\partial \Omega$; cf.~Section~\ref{sec:prel}. The eigenvalues corresponding to~\eqref{eq:Robin} form a real sequence bounded from below, which accumulates to $+ \infty$; we denote these eigenvalues by
\begin{align*}
 \lambda_1^{\theta_j} \leq \lambda_2^{\theta_j} \leq \dots, \quad j = 1, 2,
\end{align*}
where we count multiplicities. From~\eqref{eq:ineq} it follows immediately via the variational formulation of the eigenvalue problems that 
\begin{align*}
 \lambda_k^{\theta_1} \leq \lambda_k^{\theta_2}, \quad k \in \N.
\end{align*}
Our aim in this note is to show that the inequality becomes strict for all $k$,
\begin{align*}
 \lambda_k^{\theta_1} < \lambda_k^{\theta_2}, \quad k \in \N,
\end{align*}
whenever
\begin{align*}
 \theta_1 |_\omega < \theta_2 |_\omega 
\end{align*}
holds for an arbitrary nonempty, open set $\omega \subset \partial \Omega$. This observation complements various results on eigenvalue inequalities for the Laplacians with Dirichlet and Neumann or Dirichlet and Robin boundary conditions, see, e.g., the classical works~\cite{F91,LW86,P55,P52,S54} and the more recent contributions~\cite{AM12,F04,GM09,S08}. For further investigations of elliptic differential operators subject to (generalized) Robin boundary conditions and of their spectra we refer the reader to~\cite{AW03,BL07,BLLLP10,D00,DK10,GM08,G11,K11,KL12,LS04,LR12,W06} and their references. 

We wish to remark that if $\Omega$ is replaced by the (unbounded) exterior of a bounded Lipschitz domain one can show that the operators corresponding to $\theta_1$ and $\theta_2$ have the same essential spectra. In this case our result remains true for all eigenvalues below the bottom of the joint essential spectrum.

The proof of our result is carried out in Section~\ref{sec:main}. It adapts Filonov's method in~\cite{F04} and combines it with a consideration made in~\cite{BR12}, based on a unique continuation argument. Before that, in Section~\ref{sec:prel}, we discuss properties of elliptic differential operators with Robin boundary conditions on Lipschitz domains.

\section{Elliptic differential operators with Robin boundary conditions on Lipschitz domains}\label{sec:prel}

In this section we collect preliminary material on trace maps on Lipschitz domains and recall the definition of selfadjoint elliptic differential operators with Robin boundary conditions via sesquilinear forms.

Let us first fix the assumptions on the domain $\Omega$ and the differential expression~$\cL$.

\begin{assumption}\label{ass:A}
The set $\Omega \subset \R^n$, $n \geq 2$, is a bounded, connected Lipschitz domain, see, e.g.,~\cite{M00} for the standard definition. The differential expression $\cL$ on $\Omega$ is given by~\eqref{eq:diffExprIntro}, where $a_{jk}, a_j : \overline{\Omega} \to \C$ are bounded Lipschitz functions satisfying $a_{jk} (x) = \overline{a_{kj} (x)}$ for all $x \in \overline{\Omega}$, and $a : \Omega \to \R$ is measurable and bounded. Moreover,~$\cL$ is uniformly elliptic on $\Omega$, i.e., there exists $E > 0$ such that
\begin{align*}
% \label{eq:elliptic}
 \sum_{j, k = 1}^n a_{jk} (x) \xi_j \xi_k \geq E \sum_{k = 1}^n \xi_k^2, \quad x \in \overline{\Omega}, \quad \xi = (\xi_1, \dots, \xi_n)^T \in \R^n.
\end{align*}
\end{assumption}

Let us denote by $H^s (\Omega)$ and $H^s (\partial \Omega)$ the Sobolev spaces of orders $s \in \R$ on $\Omega$ and its boundary $\partial \Omega$, respectively. Here and in the following $\partial \Omega$ is equipped with the usual surface measure; cf.~\cite{M00}. Recall that there exists a unique bounded trace map from $H^1 (\Omega)$ onto $H^{1/2} (\partial \Omega)$ which extends the mapping 
\begin{align*}
 C^\infty (\overline{\Omega}) \ni u \mapsto u |_{\partial \Omega};
\end{align*}
we simply write $u |_{\partial \Omega}$ for the trace of an arbitrary $u \in H^1 (\Omega)$. Moreover, for $u \in H^1 (\Omega)$ satisfying $\cL u\in L^2(\Omega)$ (in the distributional sense) we define the conormal derivative $\frac{\partial u}{\partial \nu_\cL} |_{\partial \Omega}$ of $u$ at $\partial \Omega$ (with respect to the differential expression $\cL$) to be the unique element in $H^{- 1/2} (\partial \Omega)$ which satisfies the first Green identity
\begin{align}\label{eq:Green1}
 \sa_0 [u, v] = (\cL u, v) + \Big(\frac{\partial u}{\partial \nu_\cL} \big|_{\partial \Omega}, v |_{\partial \Omega} \Big)_{\partial \Omega}, \quad v \in H^1 (\Omega);
\end{align}
here $( \cdot, \cdot)_{\partial \Omega}$ is the (sesquilinear) duality between $H^{1/2} (\partial \Omega)$ and its dual space $H^{- 1/2} (\partial \Omega)$ and
\begin{align*}
 \sa_0 [u, v] := \int_\Omega \Big( \sum_{j,k = 1}^n a_{jk} \partial_k u \cdot \overline{\partial_j v} + \sum_{j = 1}^n \big( a_j (\partial_j u)\cdot  \overline v + \overline{a_j} u\cdot \overline{\partial_j v}\big)  + a u \overline v \Big) \d x
\end{align*}
for $u, v \in H^1(\Omega)$. 

Let us come to the definition of the operators under consideration. The assumptions on the Robin coefficient $\theta$ are the following.

\begin{assumption}\label{ass:B}
The function $\theta : \partial \Omega \to \R$ satisfies
\begin{align*}
 \theta \in L^{n - 1} (\partial \Omega) \quad \text{if}~n > 2 \quad \text{and} \quad \theta \in L^p (\partial \Omega)~\text{for~some}~p > 1 \quad \text{if}~n = 2.
\end{align*}
\end{assumption}

For $\theta$ as in Assumption~\ref{ass:B} we define a sesquilinear form $\sa_\theta$ in $L^2 (\Omega)$ by
\begin{align}\label{eq:atheta}
 \sa_\theta [u, v] := \sa_0 [u, v] + (\theta u |_{\partial \Omega}, v |_{\partial \Omega} )_{\partial \Omega}, \quad u, v \in \dom a_\theta := H^1 (\Omega).
\end{align}
It follows from Sobolev embedding theorems that $\theta u |_{\partial \Omega}$ belongs to $H^{- 1/2} (\partial \Omega)$ for each $u \in H^1 (\Omega)$, see, e.g,~\cite[Lemma~5.3]{GM09}. The form $\sa_\theta$ corresponds to a selfadjoint differential operator in $L^2 (\Omega)$, as the following proposition states. For the required material on sesquilinear forms and corresponding selfadjoint operators see Appendix~\ref{sec:App}.

\begin{proposition}\label{prop:essSpec}
Let Assumption~\ref{ass:A} and Assumption~\ref{ass:B} be satisfied. Then the sesquilinear form $\sa_\theta$ in~\eqref{eq:atheta} is symmetric, densely defined, semibounded below and closed in $L^2 (\Omega)$; the corresponding selfadjoint operator in $L^2 (\Omega)$ is given by
\begin{align}\label{eq:Atheta}
 A_\theta u = \cL u, \quad \dom A_\theta = \left\{ u \in H^1 (\Omega) : \cL u \in L^2 (\Omega), \frac{\partial u}{\partial \nu_\cL} \big|_{\partial \Omega} + \theta u |_{\partial \Omega} = 0 \right\}.
\end{align}
 In particular, for $\theta = 0$, $A_\theta = A_0$ is the selfadjoint Neumann operator associated with $\cL$. The spectrum of $A_\theta$ is bounded from below and its essential spectrum is empty; thus $\sigma (A_\theta)$ consists of isolated eigenvalues with finite multiplicities, which accumulate only to $+ \infty$.
\end{proposition}

For $\cL$ being the Laplacian a proof of Proposition~\ref{prop:essSpec} can be found in~\cite[Theorem~4.5~and~Lemma~5.3]{GM09}, where the compactness of the multiplication operator with the function $\theta$ from $H^{1/2} (\partial \Omega)$ to $H^{-1/2} (\partial \Omega)$ is employed. Since, due to Assumption~\ref{ass:A}, $\cL$ is elliptic and the coefficients of $\cL$ as well as the derivatives of the second and first order coefficients are bounded, an analogous proof can be done in the present situation. We omit the details and refer the reader to~\cite[Chapter~VI]{EE87} for a treatment of general second order elliptic differential operators in the framework of sesquilinear forms.

\section{Strict inequality of Robin eigenvalues for elliptic differential operators}\label{sec:main}

This section is devoted to our main result on strict inequality between Robin eigenvalues. We first state a simple lemma. It is inspired by the (only) Lemma in~\cite{F04} and~\cite[Proposition~2.5]{BR12}.

\begin{lemma}\label{lem:directSum}
Let Assumption~\ref{ass:A} be satisfied and let $\theta_1, \theta_2$ be functions which satisfy Assumption~\ref{ass:B} such that $\theta_1 \leq \theta_2$ on $\partial \Omega$. Moreover, let $A_{\theta_1}$ and $A_{\theta_2}$ be the corresponding operators as in~\eqref{eq:Atheta}. Suppose that $\theta_1 |_\omega < \theta_2|_\omega$ on a nonempty, open set $\omega \subset \partial \Omega$. Then 
\begin{align*}
 \dom A_{\theta_2} \cap \ker (A_{\theta_1} - \mu) = \{ 0 \}
\end{align*}
for all $\mu \in \R$. 
\end{lemma}

\begin{proof}
Let $\mu \in \R$ and $u \in \dom A_{\theta_2} \cap \ker (A_{\theta_1} - \mu)$, i.e., $u \in H^1 (\Omega)$ with $\cL u = \mu u$, and $u$ satisfies both boundary conditions 
\begin{align}\label{eq:bothBC}
 \frac{\partial u}{\partial \nu_\cL} \big|_{\partial \Omega} + \theta_1 u |_{\partial \Omega} = 0 \quad \text{and} \quad \frac{\partial u}{\partial \nu_\cL} \big|_{\partial \Omega} + \theta_2 u |_{\partial \Omega} = 0.
\end{align}
Then $(\theta_2 - \theta_1) u |_{\partial \Omega} = 0$ and, hence, $u |_\omega = 0$ by the assumption that $\theta_2 |_\omega > \theta_1 |_\omega$ on $\omega$. It follows from~\eqref{eq:bothBC} that
\begin{align}\label{eq:normDerivVanish}
 \frac{\partial u}{\partial \nu_\cL} \big|_{\omega} = - \theta_1 |_{\omega} u |_\omega = 0.
\end{align}
Let $\widetilde \Omega \supset \Omega$ be a bounded, connected Lipschitz domain such that $\partial \Omega \setminus \omega \subset \partial \widetilde \Omega$ and $\widetilde \Omega \setminus \Omega$ contains an open ball $\cO$. Let us extend the coefficients $a_{jk}, a_j$, and $a$ of $\cL$ to functions $\widetilde a_{jk}, \widetilde a_j$, and $\widetilde a$ on $\widetilde \Omega$ such that the corresponding differential expression $\widetilde \cL$ on $\widetilde \Omega$ satisfies Assumption~\ref{ass:A}. Moreover, let $\widetilde u$ be the extension by zero of $u$ to $\widetilde \Omega$. Then $u |_{\omega} = 0$ and~\eqref{eq:normDerivVanish} yield 
\begin{align*}
% \label{eq:distrCalc}
 \widetilde u \in H^1 (\widetilde \Omega) \quad \text{and} \quad \widetilde \cL \widetilde u = \mu \widetilde u,
\end{align*}
where the latter equation must be understood distributionally on $\widetilde \Omega$ and implies $\widetilde u \in H^2_{\rm loc} (\widetilde \Omega)$. On the other hand, $\widetilde u$ vanishes on the ball $\cO \subset \widetilde \Omega$. Thus a unique continuation argument yields $\widetilde u = 0$ on $\widetilde \Omega$; cf., e.g.,~\cite{W93} and the proof of~\cite[Proposition~2.5]{BR12}. Hence $u = 0$, which proves the lemma.
\end{proof}

Let us now come to the main result of this note. Under the assumptions of the previous lemma we denote by 
\begin{align*}
 \lambda_1^{\theta_j} \leq \lambda_2^{\theta_j} \leq \dots
\end{align*}
the eigenvalues of $A_{\theta_j}$, $j = 1, 2$, counted with multiplicities. 

\begin{theorem}\label{thm:main}
Let Assumption~\ref{ass:A} be satisfied and let $\theta_1, \theta_2$ be functions which satisfy Assumption~\ref{ass:B} such that $\theta_1 \leq \theta_2$ on $\partial \Omega$. Moreover, let $A_{\theta_1}$ and $A_{\theta_2}$ be the corresponding operators as in~\eqref{eq:Atheta}. Suppose that $\theta_1 |_\omega < \theta_2|_\omega$ on a nonempty, open set $\omega \subset \partial \Omega$. Then the inequality
\begin{align*}
 \lambda_k^{\theta_1} < \lambda_k^{\theta_2}, \quad k \in \N,
\end{align*}
holds.
\end{theorem}

\begin{proof}
Let $N_{\theta_1}$ and $N_{\theta_2}$ be the eigenvalue counting functions for $A_{\theta_1}$ and $A_{\theta_2}$, respectively, as in~\eqref{eq:EVcount}. Following~\eqref{eq:EVcountId} these functions can be expressed as
\begin{align*}
 N_{\theta_j} (\mu) = \max \left\{ \dim L : L~\text{subspace~of}~H^1 (\Omega), \sa_{\theta_j} [u] \leq \mu \| u \|_{L^2 (\Omega)}^2, u \in L \right\},
\end{align*}
$\mu \in \R$, where $a_{\theta_j}$ is the sesquilinear form corresponding to $\theta_j$ as in~\eqref{eq:atheta}, $j = 1, 2$. For $\mu \in \R$ let 
\begin{align*}
 F := \spann \left\{ \ker (A_{\theta_2} - \lambda) : \lambda \leq \mu \right\}.
\end{align*}
Then $\dim F = N_{\theta_2} (\mu)$ and
\begin{align*}
 \sa_{\theta_2} [u] \leq \mu \|u\|_{L^2 (\Omega)}^2, \quad u \in F.
\end{align*}
For $u \in F$ and $v \in \ker (A_{\theta_1} - \mu)$ we obtain with the help of~\eqref{eq:Green1}
\begin{align*}
 \sa_{\theta_1} [u + v] & = \sa_0 [u] + \sa_0 [v] + 2 \Real \sa_0 [v, u] + \big( \theta_1 (u + v) |_{\partial \Omega}, (u + v) |_{\partial \Omega} \big)_{\partial \Omega} \\
  & \leq \mu \int_\Omega \big( |u|^2 + |v|^2 \big) d x - ( \theta_2 u |_{\partial \Omega}, u |_{\partial \Omega})_{\partial \Omega} - ( \theta_1 v |_{\partial \Omega}, v |_{\partial \Omega})_{\partial \Omega} \\ 
  & \quad + 2 \Real \left( \int_\Omega (\cL v) \overline{u} d x + \Big( \frac{\partial v}{\partial \nu_\cL} \big|_{\partial \Omega}, u |_{\partial \Omega} \Big)_{\partial \Omega} \right) \\
   & \quad + \big( \theta_1 (u + v) |_{\partial \Omega}, (u + v) |_{\partial \Omega} \big)_{\partial \Omega},
\end{align*}
where we have used $\cL v = \mu v$. Applying this identity once more and making use of $\frac{\partial v}{\partial \nu_\cL} |_{\partial \Omega} = - \theta_1 v |_{\partial \Omega}$ it follows
\begin{align}\label{eq:Nineq}
 \sa_{\theta_1} [u + v] & \leq \mu \left( \int_\Omega \big( |u|^2 + |v|^2 \big) d x + 2 \Real \int_\Omega v \overline{u} d x \right) + \big( (\theta_1 - \theta_2) u |_{\partial \Omega}, u |_{\partial \Omega} \big)_{\partial \Omega} \nonumber \\
 & \leq \mu \left\| u + v \right\|_{L^2 (\Omega)}^2,
\end{align}
since $\theta_1 \leq \theta_2$. By Lemma~\ref{lem:directSum}, $\dim (F + \ker (A_{\theta_1} - \mu) ) = N_{\theta_2} (\mu) + \dim \ker (A_{\theta_1} - \mu)$, therefore~\eqref{eq:Nineq} yields
\begin{align}\label{eq:Nid}
 N_{\theta_1} (\mu) \geq N_{\theta_2} (\mu) + \dim \ker (A_{\theta_1} - \mu), \quad \mu \in \R.
\end{align}
Letting $k \in \N$ be arbitrary and $\mu = \lambda_k^{\theta_2}$ we obtain from~\eqref{eq:Nid}
\begin{align*}
 \# \Big\{ j \in \N : \lambda_j^{\theta_1} < \mu \Big\} = N_{\theta_1} (\mu) - \dim \ker (A_{\theta_1} - \mu) \geq N_{\theta_2} (\mu) \geq k.
\end{align*}
Hence $\lambda_k^{\theta_1} < \lambda_k^{\theta_2}$, the assertion of the theorem.
\end{proof}

\begin{appendix}

\section{Sesquilinear forms and selfadjoint operators}\label{sec:App}

In this appendix we briefly summarize basic statements on semibounded sesqui\-linear forms and corresponding selfadjoint operators. Here $\cH$ is a complex Hilbert space with inner product $(\cdot, \cdot)$ and norm $\| \cdot \|$. A sesquilinear form (short: form) in $\cH$ is a mapping $\sa : \dom \sa \times \dom \sa \to \C$ which is linear in the first and anti-linear in the second entry, where $\dom \sa$ is a linear subspace of $\cH$. We say that $\sa$ is densely defined if $\dom \sa$ is dense in $\cH$. The form $\sa$ is called symmetric if  
\begin{align*}
 \sa [u, v] = \overline{\sa [v, u]}, \quad u, v \in \dom \sa, 
\end{align*}
and semibounded below if there exists $c_\sa \in \R$ with
\begin{align*}
 \sa [u] := \sa [u, u] \geq c_\sa \|u\|^2, \quad u \in \dom \sa.
\end{align*}
Furthermore, $\sa$ is said to be closed if $\dom \sa$, equipped with the norm 
\begin{align*}
% \label{eq:formNorm}
 \| u \|_\sa := \left( \sa [u] + (1 - c_\sa) \|u\|^2 \right)^{1/2}, \quad u \in \dom \sa,
\end{align*}
is complete. The items of the following proposition can be found in several standard textbooks as, e.g.,~\cite{EE87,K95,RS78}.

\begin{proposition}\label{prop:reprThm}
Let the sesquilinear form $\sa$ in $\cH$ be densely defined, symmetric, semibounded below by some $c_\sa \in \R$, and closed. Then the following assertions hold.
\begin{enumerate}
 \item There exists a unique selfadjoint operator $A$ in $\cH$ with $\dom A \subset \dom \sa$ and
\begin{align*}
 ( A u, v) = \sa [u, v], \quad u \in \dom A, v \in \dom \sa.
\end{align*}
Moreover, the spectrum of $A$ is bounded from below by $c_\sa$.
 \item Assume, additionally, that $A$ has an empty essential spectrum and let
\begin{align*}
 \lambda_1 \leq \lambda_2 \leq \dots
\end{align*}
be the eigenvalues of $A$, counted with multiplicities. Then the min-max principle
\begin{align*}
 \lambda_j (A) = \min_{\substack{L~\text{subspace~of}~\dom \sa \\ \dim L = j}} \max_{\substack{u \in L \\ \|u\| = 1}} \sa [u], \quad j \in \N,
\end{align*}
holds. In particular, the eigenvalue counting function
\begin{align}\label{eq:EVcount}
 N_A (\mu) := \# \left\{ j \in \N : \lambda_j \leq \mu \right\}, \quad \mu \in \R,
\end{align}
can be expressed as
\begin{align}\label{eq:EVcountId}
 N_A (\mu) = \max \left\{ \dim L : L~\textup{subspace~of}~\dom \sa, \sa [u] \leq \mu \|u\|^2, u \in L \right\}.
\end{align}
\end{enumerate}
\end{proposition}

\end{appendix}

\end{document}